%% file: morphisms.tex
\documentclass[oneside,english]{amsart}
\usepackage[T1]{fontenc}
\usepackage[latin9]{inputenc}
\usepackage{amsthm}
\usepackage{amssymb}
\usepackage{graphicx}

\makeatletter
\numberwithin{equation}{section}
\numberwithin{figure}{section}
  \theoremstyle{remark}
  \newtheorem*{rem*}{\protect\remarkname}
\theoremstyle{plain}
\newtheorem{thm}{\protect\theoremname}
  \theoremstyle{plain}
  \newtheorem{lem}[thm]{\protect\lemmaname}
  \theoremstyle{plain}
  \newtheorem{cor}[thm]{\protect\corollaryname}
  \theoremstyle{plain}
  \newtheorem{prop}[thm]{\protect\propositionname}

\usepackage{lmodern}
\usepackage{tikz}
\usepackage{url}
\usepackage{tikz-qtree}
\usetikzlibrary{shapes}

\tikzset{square/.style={draw,rectangle,fill=blue!5}}
\tikzset{maybesquare/.style={draw,regular polygon,regular polygon sides=8,fill=green!10}}
\tikzset{nonsquare/.style={draw,circle}} 
\tikzset{root/.style={}}

\makeatother

\usepackage{babel}
  \providecommand{\corollaryname}{Corollary}
  \providecommand{\lemmaname}{Lemma}
  \providecommand{\propositionname}{Proposition}
  \providecommand{\remarkname}{Remark}
\providecommand{\theoremname}{Theorem}

\begin{document}

\title{Morphic Words and nested recurrence relations}

\author{Marcel Celaya and Frank Ruskey}

\address{Dept. of Mathematics and Statistics, McGill University, CANADA}

\email{marcel.celaya@mail.mcgill.ca}

\address{Dept. of Computer Science, University of Victoria, CANADA}

\email{ruskey@cs.uvic.ca}

\date{January 9, 2012}
\begin{abstract}
We explore a family of nested recurrence relations with arbitrary
levels of nesting, which have an interpretation in terms of fixed
points of morphisms over a countably infinite alphabet. Recurrences
in this family are related to a number of well-known sequences, including
Hofstadter's $G$ sequence and the Conolly and Tanny sequences. For
a recurrence $a\left(n\right)$ in this family with only finitely
terms, we provide necessary and sufficient conditions for the limit
$a\left(n\right)/n$ to exist.
\end{abstract}
\maketitle

\section{Introduction}

\global\long\def\symbol#1{\mathtt{#1}}
\global\long\def\intsymbol#1{\left[#1\right]}
\global\long\def\syms{\symbol s}
\global\long\def\iversonian#1{[\![#1]\!]}
\global\long\def\tree#1{\mathcal{#1}}
When Hofstadter described his $G$ sequence in \cite{hofstadter_douglas_r._godel_1979},
defined to be $g\left(0\right)=0$ and $g\left(n\right)=n-g\left(g\left(n-1\right)\right)$
for $n\geq1$, he mentions his discovery of a curious interpretation
of $g\left(n\right)$ in terms of an infinite rooted tree $\tree T$.
Starting with disconnected nodes labelled $1,2,3,\ldots$, The tree
$\tree T$ is constructed step-by-step as follows: step 1 places node
1 as the root and on step $n>1$, node $n$ is attached to $\tree T$
as the right-most child of node $g\left(n\right)$. Hofstadter notes
that $\tree T$ has a very interesting structure; for instance, the
number of nodes at each depth is determined by the Fibonacci sequence.
A proof of this interpretation has recently been given by Mustazee
Rahman in \cite{rahman_combinatorial_2011}.

Such a tree interpretation has been used successfully to shed light
on the behaviour of various other nested recurrences as well. In \cite{kubo_conways_1996},
Kubo and Vakil provide an elegant recursive decomposition of the tree
$\tree T$ representing the Hofstadter-Conway sequence
\[
c\left(n\right)=c\left(n-c\left(n-1\right)\right)+c\left(c\left(n-1\right)\right)\mbox{; }c\left(1\right)=c\left(2\right)=1,
\]
and use it to prove a number of interesting theorems about $c\left(n\right)$.
Also of relevance is Golomb's self describing sequence, which is the
unique increasing sequence $b\left(n\right)$ for which $b\left(1\right)=1$
and every $n\geq1$ appears $b\left(n\right)$ times. The interpretation
for $\tree T$ in the case of $b\left(n\right)$ is inherent in the
definition of $b\left(n\right)$; in $\tree T$, every child of node
$n\geq2$ has $n$ children (node 2 is a child of itself). The nested
recurrence for $b\left(n\right)$, due to Colin Mallows \cite{graham_concrete_1994},
is 
\[
b\left(n\right)=b\left(n-b\left(b\left(n-1\right)\right)\right)+1\mbox{; }b\left(1\right)=1\mbox{.}
\]
Interestingly, the discovery of this recurrence came after $b\left(n\right)$
was introduced in \cite{golomb_problem_1966}.

Hofstadter's $G$ sequence, however, serves as a canonical example
of the kind of sequences we explore in this paper. The tree $\tree T$
arising from $g\left(n\right)$, which appears as the right subtree
of the tree in Figure \ref{fig:FibT}, has a very specific structure.
In $\tree T$, every square node has two children: a square node followed
by a circle node. On the other hand, every circle node has only one
child, a square node. This suggests that $\tree T$ can be completely
described by a simple morphism, namely $\symbol 1\rightarrow\symbol{10}$
and $\symbol{0\rightarrow1}$, with $\symbol 1$'s ($\symbol 0$'s)
representing the square (circle) nodes. This is the morphism whose
unique fixed point is the Fibonacci word, a fact which is hardly a
coincidence. In this paper, we identify a family of nested recursions
which admit a similar kind of ``morphic'' tree interpretation.

\begin{figure}
\input{fibtree_symbols.tex}\caption{\label{fig:FibT}The tree $\tree T$ representing $a\left(n\right)=n-1-a\left(a\left(n-1\right)\right)$,
which is a translation of Hofstadter's $G$ sequence (A005206) \cite{oeis}.}
\end{figure}

\subsection{A note on notation.}

We adopt a number of conventions in this paper. For a proposition
$P$, we define $\iversonian P$ to be 1 when $P$ is true and 0 when
$P$ is false. For a given sequence $a\left(n\right)$, we let $a^{k}\left(n\right)$
denote $k$-fold composition and $\nabla a\left(n\right)$ denote
the backward difference $a\left(n\right)-a\left(n-1\right)$. For
a symbol $\syms$, we write $\syms^{k}$ to mean the word $\syms\syms\ldots\syms$
of length $k$. We write $\left|\symbol W\right|$ to denote the number
of symbols in the word $\symbol W$, and $\left|\symbol W\right|_{\syms}$
to denote the number of occurrences of the symbol $\syms$ in the
word $\symbol W$. To avoid ambiguity, we distinguish symbols $\symbol 0,\symbol 1,\symbol 2,\ldots$
from nonnegative integers by using square brackets; $\intsymbol{2^{3}}$
is the symbol $\mathtt{8}$ while $\intsymbol 2^{3}$ is the word
$\mathtt{222}$.

\section{Main Theorem}

Let $\Sigma,\Pi$ be alphabets (sets of letters). The set of all finite
words with letters from $\Sigma$, including the empty word $\epsilon$,
is denoted by $\Sigma^{*}$. A \emph{morphism} $\sigma:\Sigma^{*}\rightarrow\Pi^{*}$
is a function such that if $\symbol W_{1},\symbol W_{2}$ are words
in $\Sigma^{*}$, then $\sigma\left(\symbol W_{1}\symbol W_{2}\right)=\sigma\left(\symbol W_{1}\right)\sigma\left(\symbol W_{2}\right)$
and $\sigma\left(\epsilon\right)=\epsilon$. Any such morphism $\sigma$
is uniquely determined by it's action on the individual letters of
$\Sigma$. Hence, to define $\sigma$, it is sufficient to specify
$\sigma\left(\symbol x\right)$ for each letter $\symbol x\in\Sigma$. 

Let $R=\left\langle s,r_{1},r_{2},r_{3},\ldots\right\rangle $ be
a sequence of nonnegative integers with $s\geq1$. Let $\sigma_{R}$
be the following morphism on the infinite alphabet $\Sigma=\left\{ \symbol r,\symbol 0,\symbol 1,\symbol 2,\ldots\right\} $:
\begin{align*}
\symbol r & \rightarrow\symbol{r0}^{s}\\
\intsymbol j & \rightarrow\intsymbol{j+1}\symbol 0^{r_{j+1}};\quad j\geq0.
\end{align*}
Figure \ref{fig:FibT} above depicts this morphism when $R=\left\langle 1,0,1,1,1,\ldots\right\rangle $.

Let $\tree T$ be a rooted tree in which each node is labelled with
a symbol in $\Sigma$. If $V=\left\{ v_{1},\ldots,v_{k}\right\} $
is a contiguous collection of nodes all on the same row in $\tree T$,
then we say that $V$ \emph{spells} the word $\mbox{\ensuremath{\symbol w}}_{1}\ldots\symbol w_{k}\in\Sigma^{*}$
if the label of node $v_{i}$ is $\symbol w_{i}$ for $1\leq i\leq k$.

Consider the infinite rooted tree $\tree T_{R}$ in which the root
gets labelled $\symbol r$, and the children of each node labelled
$\symbol x\in\Sigma$ spell the word $\sigma_{R}\left(\symbol x\right)$.
If $v$ is a node in $\tree T_{R}$, let $\ell\left(v\right)$ count
the number of nodes to the left of $v$. Now define $a_{R}:\mathbb{Z}\rightarrow\mathbb{Z}_{\geq0}$
to be the unique function that satisfies $a_{R}\left(n\right)=0$
for $n<0$ and
\[
a_{R}\left(\ell\left(v\right)\right)=\ell\left(\mbox{parent}\left(v\right)\right)
\]
for any node $v$ in $\tree T_{R}$ which is not the root. Put another
way, if there are $n$ nodes sitting to the left of some node $v$
in $\tree T_{R}$, then $a_{R}\left(n\right)$ counts how many nodes
are sitting to the left of the parent of $v$. This quantity doesn't
depend on $v$, a fact we prove in the following lemma.
\begin{rem*}
When the sequence $R$ is understood in context, the subscripts in
$a_{R}$, $\sigma_{R}$ and $\tree T_{R}$ are omitted as is done
below.\end{rem*}
\begin{lem}
\label{lem:an-is-well-defined.}The function $a\left(n\right)$ is
well-defined.\end{lem}
\begin{proof}
The $k$\textsuperscript{th} row in $\tree T$ spells the word $\sigma^{k}\left(\symbol r\right)$
for any $k\geq0$. Moreover, $\sigma$ is \emph{nonerasing}, that
is, $\sigma\left(\symbol x\right)\neq\epsilon$ for all $\symbol x\in\Sigma$.
We also have $\symbol r\in\sigma^{k}\left(\symbol r\right)$ for any
$k$, and $\left|\sigma\left(\symbol r\right)\right|\geq2$. Thus
$n\mapsto\left|\sigma^{n}\left(\symbol r\right)\right|$ is a strictly
increasing function, which means the size of the rows in $\tree T$
is unbounded. In particular, for any $n\geq0$ there exists some node
$v$ such that $\ell\left(v\right)=n$.

We have that $\sigma$ is nonerasing and $\sigma\left(\symbol r\right)=\symbol{rZ}$
for some nonempty word $\symbol Z$, and so $\sigma$ is said to be
\emph{prolongable on }$\symbol r$ \cite[p. 10]{allouche_automatic_2003}.
Hence $\sigma^{\infty}\left(\symbol r\right)$ is a well-defined right-infinite
word. Moreover, the nodes in each row in $\tree T$ collectively spell
out a prefix of $\symbol{\sigma^{\infty}\left(\symbol r\right)}$.
Now suppose $v$ is a non-root node. Consider the word $\symbol P$
spelled by all nodes to the left of and including the parent of $v$.
This $\symbol P$ is the smallest prefix $\symbol R$ of $\sigma^{\infty}\left(\symbol r\right)$
of such that $\left|\sigma\left(\symbol R\right)\right|\geq\ell\left(v\right)+1$.
Thus, $\ell\left(\mbox{parent}\left(v\right)\right)=\left|\symbol P\right|-1$.
This quantity does not depend on $v$, merely $\ell\left(v\right)$.
\end{proof}
Throughout this paper, the word $\sigma^{\infty}\left(\symbol r\right)$
appearing in Lemma \ref{lem:an-is-well-defined.} will be very important
in our analysis of the sequence $a\left(n\right)$. Let $\symbol W$
be the right-infinite word which satisfies $\symbol{rW}=\sigma^{\infty}\left(\symbol r\right)$.
Equivalently, write $\symbol W=\symbol r^{-1}\sigma^{\infty}\left(\symbol r\right)$.
We now give two alternative but related interpretations of $a\left(n\right)$
in terms of $\symbol W$. The first says that $a\left(n\right)$ counts
the non-zero symbols in the length-$n$ prefix of $\symbol W$. The
second, which follows immediately from the first, says that the first-difference
sequence $\left\{ \nabla a\left(n\right)\right\} _{n\geq1}$ is the
binary sequence obtained by replacing all nonzero symbols in $\symbol W$
with 1.
\begin{lem}
\label{lem:Prefix}Let $\symbol W_{n}$ be the prefix of $\symbol r^{-1}\sigma^{\infty}\left(\symbol r\right)$
of length $n$. For $n\geq0$, 
\[
\left|\symbol W_{n}\right|_{\symbol 0}=n-a\left(n\right).
\]
\end{lem}
\begin{proof}
Pick a non-root node $v$ in $\tree T$ such that $\ell\left(v\right)=n$.
By definition of $\sigma$, every node to the left of and including
$v^{1}:=\mbox{parent}\left(v\right)$ has exactly one child which
is not labelled $\symbol 0$. Moreover, each such child is always
the left-most node among its siblings. There is therefore a 1-1 correspondence
between the nodes to the left of and including $v^{1}$, and the nodes
to the left of and including $v$ \emph{that are not labelled} $\symbol 0$.
The number of symbols not labelled $\symbol 0$ in $\symbol{rW}_{n}$
is therefore $\ell\left(v^{1}\right)+1$. Hence,
\[
\left|\symbol W_{n}\right|_{\symbol 0}=\left|\symbol{rW}_{n}\right|_{\symbol 0}=n+1-\left(\ell\left(v^{1}\right)+1\right)=n-a\left(\ell\left(v\right)\right)=n-a\left(n\right).
\]

\end{proof}
A \emph{coding} is a morphism $\beta:\Sigma^{*}\rightarrow\Pi^{*}$
such that $\left|\beta\left(\symbol s\right)\right|=1$ for all $\symbol s\in\Sigma$.
In the following lemma, define $\beta:\Sigma^{*}\rightarrow\left\{ \symbol 0,\symbol 1\right\} ^{*}$
to be the coding $\symbol s\mapsto\left[\iversonian{\symbol s\neq\symbol 0}\right]$.
\begin{lem}
\label{lem:binary_coding}Let $\symbol B=\symbol b_{1}\symbol b_{2}\symbol b_{3}\ldots=\beta\left(\symbol r^{-1}\sigma^{\infty}\left(\symbol r\right)\right)$
For all $n\geq1$,
\[
\symbol b_{n}=\left[\nabla a\left(n\right)\right].
\]
\end{lem}
\begin{proof}
Let $\symbol W_{n}=\symbol w_{1}\ldots\symbol w_{n}$ be as in Lemma
\ref{lem:Prefix}. Then,
\begin{align*}
\nabla a\left(n\right)=a\left(n\right)-a\left(n-1\right) & =1-\left(\left|\symbol W_{n}\right|_{\symbol 0}-\left|\symbol W_{n-1}\right|_{\symbol 0}\right)\\
 & =1-\iversonian{\symbol w_{n}=\symbol 0}\\
 & =\iversonian{\symbol w_{n}\neq\symbol 0}.
\end{align*}
Hence
\[
\left[\nabla a\left(n\right)\right]=\left[\iversonian{\symbol w_{n}\neq\symbol 0}\right]=\beta\left(\symbol w_{n}\right)=\symbol b_{n}.
\]

\end{proof}
The following theorem is the main result of this paper. This theorem,
as well as its proof, makes use of the following two doubly-indexed
quantities:
\begin{align*}
r_{i,j} & :=\iversonian{r_{i}\geq j}\\
c_{i,j} & :=r_{i,j}-r_{i-1,j}\mbox{.}
\end{align*}
Here, and throughout the rest of the paper, we define $r_{0}=-1$.
\begin{thm}
\label{thm:main}For $n<s$, $a\left(n\right)=0$ and for $n\geq s$,
\begin{equation}
a\left(n\right)=n-s-\sum_{i,j\geq1}c_{i,j}a^{i}\left(n-j\right).\label{eq:main_thm}
\end{equation}

\end{thm}
\noindent{}We'll call a recurrence \emph{morphic} if it admits an
interpretation in terms of a tree $\tree T$ as described above.

\section{Some Examples}

In this section we present what we consider to be interesting examples
of morphic recurrences.

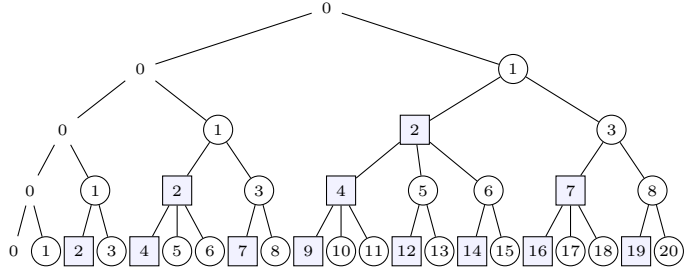
\begin{figure}
\input{root2tree.tex}

\caption{The tree $\tree T$ arising from $R=\left\langle 1,1,2,2,2,\ldots\right\rangle $,
with the value $\ell\left(v\right)$ shown on each node $v$. The
corresponding recurrence is $a\left(n\right)=n-1-a\left(n-1\right)-a\left(a\left(n-2\right)\right).$
This recurrence satisfies $a\left(n\right)=\lfloor(\sqrt{2}-1)\left(n+1\right)\rfloor$
(A097508) \cite{oeis}.}

\end{figure}

\subsection{Beatty Sequences}

A Beatty sequence is a sequence which has the form $\left\{ \left\lfloor \alpha n\right\rfloor :n\geq1\right\} $,
where $\alpha$ is some irrational constant. Of particular interest
is the case when $\alpha$ has the continued fraction expansion $\left[0;k,k,k,\ldots\right]$,
or equivalently, $\alpha=\frac{1}{2}\left(\sqrt{k^{2}+4}-k\right)$.
In this case, the corresponding Beatty sequence is the solution to
a morphic recurrence.
\begin{cor}
\label{cor:beatty}Let $k\geq1$, and let $\alpha=\left[0;k,k,k,\ldots\right]$.
Assume $a\left(n\right)=0$ for $n<k$, and for $n\geq k$, let
\[
a\left(n\right)=n-k+1-\left(\sum_{i=1}^{k-1}a\left(n-i\right)\right)-a\left(a\left(n-k\right)\right)\mbox{.}
\]
Then for nonnegative $n$, $a\left(n\right)=\left\lfloor \alpha\left(n+1\right)\right\rfloor $.\end{cor}
\begin{proof}
When $k=1$, this recurrence is simply Hofstadter's $G$ sequence,
and the conclusion has been proven independently by several authors
\cite{burton_curious_1986,downey_family_1982,granville_strange_1988}.
For $k>1$, we consider the sequence $R=\left\langle s,r_{1},r_{2},\ldots\right\rangle $
which gives rise to this recurrence. A little calculation shows this
recurrence arises precisely when $s=r_{1}=k-1$ and $r_{i}=k$ for
$i\geq2$. The corresponding morphism $\sigma$ is
\begin{align*}
\symbol r & \rightarrow\symbol{r0}^{k-1}\\
\symbol 0 & \rightarrow\symbol 1\symbol 0^{k-1}\\
\symbol 1 & \rightarrow\symbol 1\symbol 0^{k},
\end{align*}
where, for brevity, the symbols $\symbol 1,\symbol 2,\symbol 3,\ldots$
are all identified as $\symbol 1$. This identification is not a problem,
since the underlying structure of the tree $\tree T$ remains the
same.

Let $\symbol C$ be the infinite word $\symbol c_{1}\symbol c_{2}\symbol c_{3}\ldots$,
where $\symbol c_{n}=\intsymbol{\left\lfloor \alpha\left(n+1\right)\right\rfloor -\left\lfloor \alpha n\right\rfloor }$.
Applying a theorem of A. A. Markov, Stolarsky showed in \cite{stolarsky_beatty_1976}
that $\symbol C=\gamma^{\infty}\left(\symbol 0\right)$, where $\gamma$
is the morphism $\symbol 0\rightarrow\symbol 0^{k-1}\symbol 1$, $\symbol 1\rightarrow\symbol 0^{k-1}\symbol{10}$.
The two morphisms $\sigma$ and $\gamma$ appear to be similar in
their action on $\symbol 0$ and $\symbol 1$. Indeed, for any right-infinite
word $\symbol W$ on $\left\{ \symbol 0,\symbol 1\right\} $, 
\begin{equation}
\gamma\left(\symbol W\right)=\symbol 0^{k-1}\sigma\left(\symbol W\right)\mbox{.}\label{eq:sigma_gamma}
\end{equation}
This is because if $\symbol W=\symbol w_{1}\symbol w_{2}\symbol w_{3}\ldots$,
and $\phi$ is the morphism $\symbol{1\rightarrow10}$, $\symbol 0\rightarrow\symbol 1$,
\begin{align*}
\gamma\left(\symbol W\right) & =\gamma\left(\symbol w_{1}\right)\gamma\left(\symbol w_{2}\right)\gamma\left(\symbol w_{3}\right)\ldots\\
 & =\symbol 0^{k-1}\phi\left(\symbol w_{1}\right)\symbol 0^{k-1}\phi\left(\symbol w_{2}\right)\symbol 0^{k-1}\phi\left(\symbol w_{3}\right)\ldots\\
 & =\symbol 0^{k-1}\sigma\left(\symbol w_{1}\right)\sigma\left(\symbol w_{2}\right)\sigma\left(\symbol w_{3}\right)\ldots\\
 & =\symbol 0^{k-1}\sigma\left(\symbol W\right).
\end{align*}

Define the infinite word $\symbol B:=\symbol b_{1}\symbol b_{2}\symbol b_{3}\ldots$
so that $\symbol b_{n}=\left[\nabla a\left(n\right)\right]$. By Lemma
\ref{lem:binary_coding}, $\symbol{rB}=\sigma^{\infty}\left(\symbol r\right)$,
and so
\[
\symbol{rB}=\sigma\left(\symbol{rB}\right)=\sigma\left(\symbol r\right)\sigma\left(\symbol B\right)=\symbol{r0}^{k-1}\sigma\left(\symbol B\right).
\]
This implies
\begin{align*}
\symbol B & =\symbol 0^{k-1}\sigma\left(\symbol B\right)=\gamma\left(\symbol B\right).
\end{align*}
As $\gamma$ has only one fixed point, 
\[
\symbol B=\gamma^{\infty}\left(\symbol 0\right)=\symbol C.
\]

\end{proof}

\subsection{$k$-ary Meta-Fibonacci Sequences}

Define the recurrence
\[
b\left(n\right)=\sum_{i=1}^{k}b\left(n-i-b\left(n-i\right)\right)+s,
\]
which has initial conditions $b\left(n\right)=\max\left(0,n\right)$
for $n<s$ and is parametrized by two constants $k,s\geq1$.

Recurrences similar to this one have been studied extensively in recent
years \cite{isgur_solving_2011,ruskey_combinatorics_2009}. In fact,
$b\left(n\right)$ is a special case of those which are described
in \cite{isgur_solving_2011}, and from there a combinatorial interpretation
is given for $b\left(n\right)$ in terms of an infinite rooted tree
not unlike the one presented in this paper.

It turns out that $b\left(n\right)$ can also be described in terms
of a morphic recurrence.
\begin{cor}
If $a\left(n\right)$ is the sequence $n-b\left(n\right)$, then for
$n\geq s$, $a\left(n\right)$ satisfies
\[
a\left(n\right)=n-s-\sum_{i=1}^{k}a\left(n-i\right)+\sum_{i=1}^{k}a\left(a\left(n-i\right)\right)\mbox{,}
\]
which is the (\ref{eq:main_thm}) recurrence arising from letting
$R=\left\langle s,k,0,0,0,\ldots\right\rangle $.\end{cor}
\begin{proof}
The recurrence for $a\left(n\right)$ follows directly from the definitions
of $a\left(n\right)$ and $b\left(n\right)$. We illustrate this calculation
when $s=1$ and $k=2$, so that 
\begin{align}
a\left(n\right) & =n-1-a\left(n-1\right)-a\left(n-2\right)+a\left(a\left(n-1\right)\right)+a\left(a\left(n-2\right)\right)\label{eq:complement_conolly}
\end{align}
with $a\left(n\right)=0$ for $n\leq1$. If we consider the sequence
$b\left(n\right)=n-a\left(n\right)$, then we have for $n\geq1$,
\begin{align*}
b\left(n\right) & =n-\left(n-1-a\left(n-1\right)+a\left(a\left(n-1\right)\right)-a\left(n-2\right)+a\left(a\left(n-2\right)\right)\right)\\
 & =1+a\left(n-1\right)-a\left(a\left(n-1\right)\right)+a\left(n-2\right)-a\left(a\left(n-2\right)\right)\\
 & =1+b\left(a\left(n-1\right)\right)+b\left(a\left(n-2\right)\right)\\
 & =1+b\left(n-1-b\left(n-1\right)\right)+b\left(n-2-b\left(n-2\right)\right)\mbox{.}
\end{align*}

\end{proof}
It can be directly shown that when $s=1$ and $k=2$, $b\left(n\right)$
is equal to $c\left(n+1\right)-1$, where $c\left(n\right)$ is the
well-known Conolly sequence (A046699) \cite{oeis} introduced in \cite{vajda_fibonacci_1989}.

\begin{figure}
\input{conollytree.tex}\caption{\label{fig:Conolly}The tree $\tree T$ representing the sequence
(\ref{eq:complement_conolly}), with the left subtree, itself a copy
of $\tree T$, stubbed out. As before, $\ell\left(v\right)$ is shown
for each node $v$.}

\end{figure}
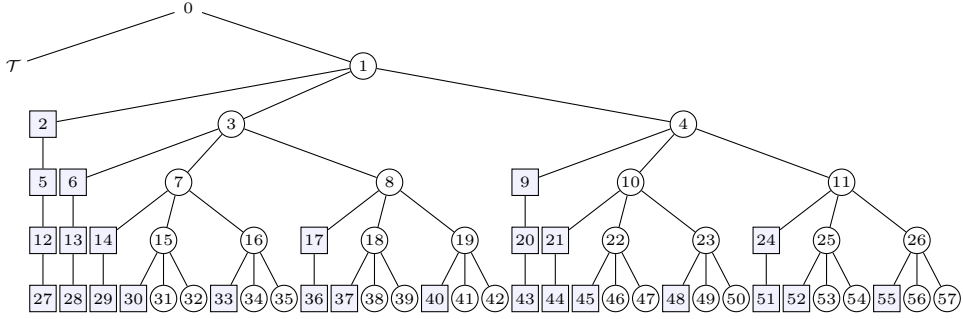

\section{Proof of the Theorem}

\global\long\def\parent#1#2{#1^{#2}}
We define a \emph{left-most} node in $\tree T$ to be a node which
has no siblings on its left. We also define an $n$\emph{-node} in
$\tree T$ to be a node $v$ which satisfies $\ell\left(v\right)=n$.
Finally, we denote the parent of a node $v$ by $\parent v1$, the
grandparent by $\parent v2$, and so on. We assume $v=\parent v0$.

To prove the theorem, we introduce the sequence 
\[
d\left(n\right)=\max\left\{ k:a^{k}\left(n\right)-a^{k}\left(n-1\right)=1\right\} 
\]
for $n\geq1$. It is a consequence of Lemma \ref{lem:binary_coding}
that $a\left(n\right)$ is \emph{slow-growing}; that is, $a$ is a
monotone increasing sequence with successive differences equal to
either zero or one. For each $n$, the number zero lies in this set.
Moreover, since $a\left(0\right)=0$, the set is bounded above by
the depth of any arbitrary $n$-node. Therefore, $d\left(n\right)$
is well-defined for all $n$. If we consider an $n$-node and its
adjacent $\left(n-1\right)$-node in $\tree T$, we may interpret
$d\left(n\right)$ as the length of the path from either of these
nodes to their lowest common ancestor.

We begin with the following easy lemma:
\begin{lem}
\label{lem:square}Let $v$ be an $n$-node. The following statements
are equivalent:\end{lem}
\begin{enumerate}
\item $v$ is labelled $\intsymbol k$ for some $k\geq1$.
\item $v$ is a left-most node, and $n\geq1$.
\item $\nabla a\left(n\right)=1$.
\item $d\left(n\right)\geq1$.
\end{enumerate}
Nodes which satisfy these conditions will be called \emph{square}
nodes.
\begin{proof}
$\left(1\right)\implies\left(2\right)$: If $v$ is labelled $\left[k\right]$
for $k\geq1$, then $\parent v1$ must exist and must be labelled
$\left[k-1\right]$. The child nodes of $\parent v1$ spell the word
$\sigma\left(\left[k-1\right]\right)=\left[k\right]\symbol 0^{r_{k}}$,
and so $v$, which has the unique $\left[k\right]$ label, must be
the left-most child of $\parent v1$. The fact that $n>0$ follows
from the fact that all 0-nodes have the label $\symbol r$, not $\left[k\right]$.

$\left(2\right)\implies\left(3\right)$: Node $\parent v1$ must exist
since otherwise $v$ would be a 0-node. For the same reason, $\parent v1$
must not be a 0-node. But this means that $\parent v1$ is positioned
to the right of another node $\parent w1$, and since $\sigma$ is
nonerasing, every node in $\tree T$ has a child, and so $\parent w1$
is the parent of a $\left(n-1\right)$-node $w$ sitting next to $v$.
Therefore, we have
\[
\nabla a\left(n\right)=a\left(n\right)-a\left(n-1\right)=\ell\left(\parent v1\right)-\ell\left(\parent w1\right)=1.
\]

$\left(3\right)\implies\left(4\right)$: This implication is immediate
from the definitions.

$\left(4\right)\implies\left(1\right)$: The fact that $d\left(n\right)\geq1$
means that, in particular, $a\left(n\right)-a\left(n-1\right)=1$.
Since $a$ is nonnegative, $a\left(n\right)$ must be positive, which
implies that $\parent v1$ exists and $\ell\left(\parent v1\right)\geq1$.
In other words, $\parent v1$ is not the left-most node in its row
in $\tree T$. This implies that the label of $\parent v1$ is $\left[j\right]$
for some $j\geq0$, since the only nodes with the $\symbol r$ label
are the 0-nodes. This also implies that $v$ is not a 0-node, since
the children of the nodes to the left of $\parent v1$ must come before
$v$. In particular, there exists an $\left(n-1\right)$-node $w$
sitting next to $v$ on the same level as $v$. The fact that $\nabla a\left(n\right)=1$
implies that $v$ is the left-most node of $\parent v1$, since if
it wasn't then $w$ would also be a child of $\parent v1$ and we'd
have $a\left(n\right)=\ell\left(\parent v1\right)=a\left(n-1\right)$.
Therefore, node $v$ must have label $\left[j+1\right]$.\end{proof}
\begin{lem}
\label{lem:r_d(n)+1}Suppose $v$ is a square $n$-node. Then $\parent v1$
has $r_{d\left(n\right)}+1$ children.\end{lem}
\begin{proof}
We show that $\parent v1$ is labelled $\intsymbol{d\left(n\right)-1}$.
In fact, we show the stronger statement that the label of $\parent vi$
is $\left[d\left(n\right)-i\right]$ for $1\leq i\leq d\left(n\right)$.
The result will follow from the fact that the children of $\parent v1$
spell the word $\sigma\left(\intsymbol{d\left(n\right)-1}\right)=\intsymbol{d\left(n\right)}\symbol 0^{r_{d\left(n\right)}}$,
which has $r_{d\left(n\right)}+1$ symbols.

Let $w$ be the $\left(n-1\right)$-node to the left of $v$. Our
initial goal is to prove that node $\parent v{d\left(n\right)}$ is
labelled $\symbol 0$. By definition of $d\left(n\right)$, we have
$\parent v{d\left(n\right)}\neq\parent w{d\left(n\right)}$ but $\parent v{d\left(n\right)+1}=\parent w{d\left(n\right)+1}$.
Therefore, $\parent v{d\left(n\right)}$ and $\parent w{d\left(n\right)}$
are siblings. Since $\parent v{d\left(n\right)}$ sits to the right
of $\parent w{d\left(n\right)}$, $\parent v{d\left(n\right)}$ must
have the label $\symbol 0$. This is because for any $\symbol x\in\Sigma$,
the word $\sigma\left(\symbol x\right)$ contains only zeros after
the first symbol.

Our next goal is to show, for $1\leq i\leq d\left(n\right)-1$, that
the label of node $v^{d\left(n\right)-i}$ is $\intsymbol i$. Observe
that nodes $\parent v{d\left(n\right)-i}$ and $\parent w{d\left(n\right)-i}$
must have different parents. Indeed, if this were not the case, then
we'd have $\parent v{d\left(n\right)}=\parent w{d\left(n\right)}$,
a contradiction. As there are no nodes sitting between $\parent v{d\left(n\right)-i}$
and $\parent w{d\left(n\right)-i}$, $\parent v{d\left(n\right)-i}$
must be the left-most node of its parent, $\parent v{d\left(n\right)-\left(i-1\right)}$.
If we assume inductively that $\parent v{d\left(n\right)-\left(i-1\right)}$
is labelled $\intsymbol{i-1}$, then $\parent v{d\left(n\right)-i}$
is labelled the first symbol of $\sigma\left(\intsymbol{i-1}\right)$,
which is $\intsymbol i$.
\end{proof}
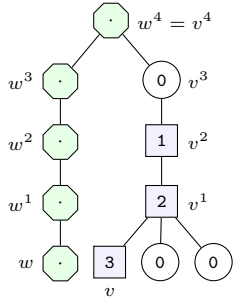
\begin{figure}
\input{second_lemma.tex}\caption{A depiction of Lemma \ref{lem:r_d(n)+1} when $d\left(n\right)=3$
and $r_{3}=2$. The green-shaded octagonal nodes may or may not be
square.}
\end{figure}

\begin{lem}
\label{lem:r_d(n-i,i)=00003D1} Suppose $n>s$ and $i\geq0$. Let
$v$ be an $n$-node and $w$ the $\left(n-i\right)$-node on the
same level as $v$. Then $r_{d\left(n-i\right),i}=1$ if and only
if $w$ is the left-most child of $v^{1}$.\end{lem}
\begin{proof}
By definition, $r_{d\left(n-i\right),i}=1$ means that $d\left(n-i\right)\geq1$
and $r_{d\left(n-i\right)}\geq i$. We therefore have that $w$ is
a square node, so by Lemma \ref{lem:r_d(n)+1}, $\parent w1$ has
at least $i+1$ children. Hence, $\parent w1$ has enough children
so that $v$ is included among them.

Conversely, assume $w$ is the left-most child of $v^{1}$. Since
$n>s$, $v^{1}$ is not a 0-node. Since $v$ and $w$ are siblings,
we further have that $n-i\geq1$. Thus, $w$ is a square node. By
Lemma \ref{lem:r_d(n)+1}, $w^{1}$ has $r_{d\left(n-i\right)}+1$
children, and by assumption, $w^{1}$ has at least $i+1$ children;
this gives $r_{d\left(n-i\right)}\geq i$. By Lemma \ref{lem:square},
$d\left(n-i\right)\geq1$. These two inequalities imply $r_{d\left(n-i\right),i}=1$.
\end{proof}

\begin{proof}
[Proof of theorem.]

We proceed by induction on $n$. When $n=s$, we expect that $a\left(n\right)=0$,
and indeed this is the case:
\[
a\left(r_{0}\right)=0-\sum_{i,j\geq1}c_{i,j}a^{i}\left(r_{0}-j\right)=0\mbox{.}
\]
Suppose, then, that $n>s$ and $a\left(n-1\right)$ satisfies the
recurrence. Then,
\begin{align*}
a\left(n\right) & =\nabla a\left(n\right)+a\left(n-1\right)=\nabla a\left(n\right)+n-1-s-\sum_{i,j\geq1}c_{i,j}a^{i}\left(n-1-j\right)\mbox{.}
\end{align*}
 Define, for $i,j\geq0$, 
\[
e_{i,j}=a^{i}\left(n-j\right)-a^{i}\left(n-1-j\right)\mbox{.}
\]
Observe that the sequence $\left(e_{1,j},e_{2,j},e_{3,j},\ldots\right)$
consists of a finite number of consecutive ones, followed by an infinite
number of consecutive zeros. Moreover, the number of ones in this
sequence is $d\left(n-j\right)$. Thus,
\begin{align*}
a\left(n\right) & =\nabla a\left(n\right)+n-1-s-\sum_{i,j\geq1}c_{i,j}\left(a^{i}\left(n-j\right)-e_{i,j}\right)\\
 & =\nabla a\left(n\right)+n-1-s+\sum_{i,j\geq1}c_{i,j}e_{i,j}-\sum_{i,j\geq1}c_{i,j}a^{i}\left(n-j\right)\mbox{.}
\end{align*}
But
\[
\sum_{i,j\geq1}c_{i,j}e_{i,j}=\sum_{j\geq1}\sum_{i=1}^{d\left(n-j\right)}c_{i,j}=\sum_{j\geq1}\sum_{i=1}^{d\left(n-j\right)}\left(r_{i,j}-r_{i-1,j}\right)=\sum_{j\geq1}r_{d\left(n-j\right),j}\mbox{,}
\]
and since 
\[
r_{d\left(n\right),0}=\iversonian{r_{d\left(n\right)}\geq0}=\iversonian{d\left(n\right)\geq1}=\nabla a\left(n\right)\mbox{,}
\]
we can write $a\left(n\right)$ as 
\[
a\left(n\right)=n-1-s+\sum_{i\geq0}r_{d\left(n-i\right),i}-\sum_{i,j\geq1}c_{i,j}a^{i}\left(n-j\right)\mbox{.}
\]
Lemma \ref{lem:r_d(n-i,i)=00003D1} provides both the existence and
uniqueness of a nonnegative integer $i$ such that $r_{d\left(n-i\right),i}=1$.
Consequently, 
\[
\sum_{i\geq0}r_{d\left(n-i\right),i}=1\mbox{.}
\]
Substituting this into the above expression for $a\left(n\right)$,
we get
\[
a\left(n\right)=n-s-\sum_{i,j\geq1}c_{i,j}a^{i}\left(n-j\right)\mbox{.}
\]

\end{proof}

\section{Generating Functions}

Let $R=\left\langle s,r_{1},r_{2},\ldots\right\rangle $ with $s\geq1$,
and write $\sigma=\sigma_{R}$. Let $\symbol L_{n}=\sigma^{n}\left(\symbol 0^{s}\right)$
and $\symbol T_{n}=\symbol r^{-1}\sigma^{n}\left(\symbol r\right)$.
One can easily show by induction that $\symbol T_{n}$ can be also
written as follows:
\[
\symbol T_{n}=\symbol L_{0}\symbol L_{1}\symbol L_{2}\ldots\symbol L_{n-1}=\symbol 0^{s}\sigma\left(\symbol 0^{s}\right)\sigma^{2}\left(\symbol 0^{s}\right)\ldots\sigma^{n-1}\left(\symbol 0^{s}\right).
\]
Given $\intsymbol x\in\Sigma\setminus\left\{ \symbol r\right\} $,
we obtain the following recurrence relation from the definition of
$\sigma$:
\[
\left|\symbol L_{n+1}\right|_{\left[x\right]}=\begin{cases}
s & \mbox{if }n=0\\
\left|\symbol L_{n}\right|_{\left[x-1\right]} & \mbox{if }n\geq1\mbox{ and }x\geq1\\
\sum_{j\geq0}r_{j+1}\left|\symbol L_{n}\right|_{\left[j\right]} & \mbox{if }n\geq1\mbox{ and }x=0.
\end{cases}
\]
From this recurrence relation we deduce that for $n\geq1$,
\[
\left|\symbol L_{n+1}\right|_{\symbol 0}=\sum_{i=0}^{n}r_{i+1}\left|\symbol L_{n-i}\right|_{\symbol 0}.
\]
Define the generating functions
\begin{align*}
R\left(z\right) & :=\sum_{n\geq0}r_{n}z^{n}\\
N\left(z\right) & :=\sum_{n\geq0}\left|\symbol L_{n}\right|_{\symbol 0}z^{n}
\end{align*}
Note that
\begin{align*}
N\left(z\right)R\left(z\right) & =\left(\sum_{n\geq0}\left|\symbol L_{n}\right|_{\symbol 0}z^{n}\right)\left(\sum_{n\geq0}r_{n}z^{n}\right)\\
 & =\sum_{n\geq0}\left(\sum_{i=0}^{n}r_{i}\left|\symbol L_{n-i}\right|_{\symbol 0}\right)z^{n}\\
 & =\sum_{n\geq0}\left(-\left|\symbol L_{n}\right|_{\symbol 0}+\sum_{i=1}^{n}r_{i}\left|\symbol L_{n-i}\right|_{\symbol 0}\right)z^{n}\\
 & =-N\left(z\right)+\sum_{n\geq0}\left(\sum_{i=0}^{n-1}r_{i+1}\left|\symbol L_{n-1-i}\right|_{\symbol 0}\right)z^{n}\\
 & =-N\left(z\right)+\sum_{n\geq1}\left|\symbol L_{n}\right|_{\symbol 0}z^{n}\\
 & =-N\left(z\right)+N\left(z\right)-s\\
 & =-s.
\end{align*}
Thus,
\[
N\left(z\right)=-\frac{s}{R\left(z\right)}.
\]
Let $L\left(z\right):=\sum_{n\geq0}\left|\symbol L_{n}\right|z^{n}$.
This is the generating function for the number of nodes of the $\left(n+1\right)$\textsuperscript{th}
row of $\tree T_{R}'$, where $\tree T_{R}'$ is the tree that results
from pruning the left-most subtree of $\tree T_{R}$. The recurrence
relation gives $\left|\symbol L_{n}\right|=\sum_{i=0}^{n}\left|\symbol L_{n}\right|_{\intsymbol i}=\sum_{i=0}^{n}\left|\symbol L_{i}\right|_{\symbol 0},$
thus
\[
L\left(z\right)=-\frac{s}{\left(1-z\right)R\left(z\right)}.
\]
Similarly, let $T\left(z\right):=\sum_{n\geq0}\left|\symbol T_{n}\right|z^{n}$,
which is the generating function for the total number of nodes of
row $n$ in $\tree T_{R}$ excluding the left-most node labelled $\symbol r$.
Since $\left|\symbol T_{n}\right|=\sum_{i=0}^{n-1}\left|\symbol L_{n}\right|$,
we have
\[
T\left(z\right)=-\frac{sz}{\left(1-z\right)^{2}R\left(z\right)}.
\]

\subsection{Recurrences with finitely many terms}

The examples of morphic recurrences analyzed thus far have only finitely
many terms. Saying that a recurrence $a_{R}\left(n\right)$ of the
form (\ref{eq:main_thm}) has finitely many terms is equivalent to
saying that $R=\left\langle s,r_{1},r_{2,}\ldots\right\rangle $ is
eventually constant; this follows directly from the statement of Theorem
\ref{thm:main}. If $k$ is the largest integer such that $r_{k}\neq r_{k+1}$,
then the generating function $R\left(z\right)$ is rational. Indeed,
\[
R\left(z\right)=r_{0}+r_{1}z+\ldots+r_{k}z^{k}+\frac{r_{k+1}z^{k+1}}{1-z}.
\]
It follows that $T\left(z\right)$ is also rational, and in particular
\[
T\left(z\right)=-\frac{sz}{\left(1-z\right)q\left(z\right)},
\]
where
\begin{equation}
q\left(z\right)=\left(1-z\right)\left(r_{0}+r_{1}z+\ldots+r_{k}z^{k}\right)+r_{k+1}z^{k+1}.\label{eq:gf_poly}
\end{equation}

\subsection{An example of a recurrence with infinitely many terms}

Theorem \ref{thm:main} implies that nested recurrences can have infinitely
many terms and still be well-defined. Suppose, for instance, that
$R=\left\langle 1,1,2,3,4,\ldots\right\rangle $. The recurrence given
by this sequence is
\[
a\left(n\right)=n-1-a\left(n-1\right)-a\left(a\left(n-2\right)\right)-a\left(a\left(a\left(n-3\right)\right)\right)-\ldots
\]
with $a\left(n\right)=0$ for $n\leq0$.

For this particular example we have $R\left(z\right)=-1+z/\left(1-z\right)^{2}$,
and hence
\[
T\left(z\right)=\frac{z}{1-3z+z^{2}}.
\]
The coefficients of this generating function are $0,1,3,8,21,55,\ldots$,
the even Fibonacci numbers starting with $F_{0}=0$. Hence, the length
of the word $\symbol T_{n}$ is $F_{2n}$. We can use this observation
to prove the following statement:
\begin{prop}
Let $F_{n}$ denote the $n$\textsuperscript{th} Fibonacci number.
For $n\geq1$,
\[
a\left(F_{2n}\right)=F_{2n-2}.
\]
\end{prop}
\begin{proof}
Let $v$ be the right-most node in $\tree T$ on row $n$. There are
$\left|\symbol T_{n}\right|$ nodes to the left of $v$. The parent
of $v$ is also a right-most node in $\tree T$, and so there are
$\left|\symbol T_{n-1}\right|$ nodes to the left of it. Then by definition,
\[
a\left(F_{2n}\right)=a\left(\left|\symbol T_{n}\right|\right)=\left|\symbol T_{n-1}\right|=F_{2n-2}.
\]

\end{proof}
Somewhat mysteriously, it also appears that $a\left(n\right)$ has
the same shifting property on the odd Fibonacci numbers; it seems
$a\left(F_{2n+1}\right)=F_{2n-1}$ for $n\geq1$. We do not have a
proof of this claim.

\section{Asymptotics}

In this section, we analyze the asymptotics of recurrences of the
form (\ref{eq:main_thm}) which have finitely many terms. In particular,
we give sufficient and necessary conditions for such a recurrence
to be asymptotically linear and determine its limiting slope. This
question has been looked at by Kiss and Zay \cite{kiss_generalization_1992}
in the particular case when $R=\left\langle 1,0,\ldots,0,1,1,\ldots\right\rangle $
with $k$ zeros. They show that $\lim_{n\rightarrow\infty}a_{R}\left(n\right)/n$
equals the unique positive root of the polynomial $x^{k}+x-1$. 

When $a_{R}\left(n\right)$ has finitely many terms, that is, when
it is the case that there exists some $k$ such that $r_{i}=r_{k+1}$
for all $i\geq k+1$, then, as done previously in this paper, we can
indentify the symbols $\left[k\right],\left[k+1\right],\left[k+2\right],\ldots$
as just one symbol, namely $\left[k\right]$. This allows us to define
$\sigma_{R}$ on a \emph{finite} alphabet $\Sigma_{k}:=\left\{ \symbol r,\symbol 0,\symbol 1,\symbol 2,\ldots,\left[k\right]\right\} $
as follows:
\begin{align*}
\symbol r & \rightarrow\symbol{r0}^{s}\\
\intsymbol j & \rightarrow\intsymbol{j+1}\symbol 0^{r_{j+1}};\quad0\leq j<k\\
\intsymbol k & \rightarrow\intsymbol k0^{r_{k+1}}.
\end{align*}
As before, the underlying structure of $\tree T_{R}$, upon which
the definition of $a_{R}\left(n\right)$ is based, remains unchanged
in this alternative definition of $\sigma_{R}$ which we now use.
\begin{thm}
Let $R=\left\langle s,r_{1},r_{2},\ldots\right\rangle ,$ $s\geq1$
be a nonnegative integer sequence which is eventually constant; that
is, there exists some $k\geq1$ such that $r_{i}=r_{k+1}$ for all
$i\geq k+1$. Then the limit
\[
\lim_{n\rightarrow\infty}\frac{a_{R}\left(n\right)}{n}
\]
exists if and only if at least one of the following two conditions
holds:
\begin{enumerate}
\item $r_{1}+r_{2}+r_{3}+\ldots=0$ or $1$.
\item $\gcd\left\{ i\geq1:r_{i}\geq1\right\} =1$.
\end{enumerate}
If it exists, it is equal to the smallest positive root of the polynomial
\[
\left(1-x\right)\left(r_{0}+r_{1}x+\ldots+r_{k}x^{k}\right)+r_{k+1}x^{k+1}.
\]

\end{thm}

\subsection{Results on nonnegative matrices}

Before giving the proof, we review some preliminaries on nonnegative
matrices. See \cite[Ch. 8]{allouche_automatic_2003} or \cite[Ch. 8]{meyer_matrix_2000}
for further details. A nonnegative square matrix $M$ is said to be
\emph{reducible} if there exists square matrices $A$ and $B$, possibly
of different sizes, and a permutation matrix $P$ such that 
\[
PMP^{T}=\begin{pmatrix}A & C\\
0 & B
\end{pmatrix},
\]
where $C$ is an arbitrary matrix and $0$ is a zero matrix. $M$
is said to be \emph{irreducible} if it is not reducible. If $M$ is
integral and interpreted as the adjancency matrix of a digraph $D$,
then a sufficient and necessary condition for the irreducibility of
$M$ is that $D$ is strongly connected \cite[p. 671]{meyer_matrix_2000}.

Denote the characteristic polynomial of a matrix $M$ by $p_{M}\left(x\right)$.
If $M$ is reducible with $A,B$ as above, then
\[
p_{M}\left(x\right)=p_{A}\left(x\right)p_{B}\left(x\right).
\]
Hence, the eigenvalues of $M$ are just the combined eigenvalues of
$A$ and $B$.

Given a nonnegative square matrix $M$, there exists an eigenvalue
$\lambda$ called the Perron-Frobenius eigenvalue which is equal to
the largest modulus of all eigenvalues of $M$. That $\lambda$ is
itself an eigenvalue is a consequence of the Perron-Frobenius theorem.
If $M$ is irreducible with Perron-Frobenius eigenvalue $\lambda$,
then there exists a positive integer $h$ such that the collection
\[
\left\{ \lambda,\lambda\omega,\lambda\omega^{2},\ldots,\lambda\omega^{h-1}\right\} ,\mbox{ }\omega:=e^{2\pi i/h},
\]
is the collection of all eigenvalues of $M$ with modulus $\lambda$.
Moreover, every eigenvalue in this collection is simple \cite[Theorem 8.3.10]{allouche_automatic_2003}.
The number $h$ is called the \emph{index of imprimitivity} \emph{of}
$M$. If $p_{M}\left(x\right)$ is written as 
\[
p_{M}\left(x\right)=c_{n}+c_{n-1}x+c_{n-2}x^{2}+\ldots+c_{1}x^{n-1}+x^{n},
\]
then $h=\gcd\left\{ j:c_{j}\neq0\right\} $ \cite[Theorem 8.3.9]{allouche_automatic_2003}.

An even stronger notion than irreducibility is that of primitivity.
A nonnegative square matrix $M$ is said to be \emph{primitive }if
there exists some integer $n\geq1$ such that $M^{n}$ has only positive
entries. If $M$ is integral and interpreted as the adjancency matrix
of a digraph $D$, then a sufficient and necessary condition for the
primitivity of $M$ is that there exists an integer $n$ such that
between any two vertices $u,v$ of $D$ there exists a walk starting
at $u$ and ending at $v$ that has length $n$.

A key property of primitive matrices is that the Perron-Frobenius
eigenvalue $\lambda$ of $M$ strictly dominates in modulus all other
eigenvalues of $M$. All primitive matrices are irreducible, and the
index of imprimitivity of a primitive matrix is 1. The converse is
also true; an irreducible matrix with an index of imprimitivity of
1 is necessarily primitive \cite[Theorem 8.3.10]{allouche_automatic_2003}.

The proof of the theorem in this section relies on a result due to
K. Saari \cite{saari_frequency_2006}, part of which is given in the
following proposition. For any morphism $\gamma:\Gamma^{*}\rightarrow\Gamma^{*}$
defined on a finite alphabet $\Gamma=\left\{ \symbol s_{1},\ldots,\symbol s_{k}\right\} $,
we define the \emph{incidence matrix of $\gamma$}, denoted $M_{\gamma}$,
to be the matrix 
\[
\left(M_{\gamma}\right)_{i,j}:=\left|\gamma\left(\symbol s_{i}\right)\right|_{\syms_{j}}.
\]

\begin{prop}
\label{prop:Saari}Let $\gamma$ be a nonerasing morphism on a finite
alphabet $\Gamma$ such that:
\begin{enumerate}
\item There exists an $\symbol s\in\Gamma$ such that $\gamma\left(\symbol s\right)=\symbol{sZ}$
for some nonempty word $\symbol Z$.
\item In the Jordan canonical form of $M_{\gamma}$, there is a Jordan block
associated with the Perron-Frobenius eigenvalue $\lambda$ of $M_{\sigma}$
which is strictly larger in dimension than any other Jordan block
associated with an eigenvalue of modulus $\lambda$.
\end{enumerate}
Then letting $\symbol W_{n}$ denote the length-$n$ prefix of $\gamma^{\infty}\left(\symbol s\right)$,
the limit
\[
\lim_{n\rightarrow\infty}\frac{\left|\symbol W_{n}\right|_{\symbol t}}{n}
\]
exists for all $\symbol t\in\Gamma$.
\end{prop}

\subsection{Proof of theorem}

To simplify the proof of the theorem a little, we only consider the
case when $s=1$. The arguments are the same in the general case.
As before we drop the subscripts from $a_{R}\left(n\right)$ and $\sigma_{R}$. 
\begin{proof}
We begin by proving that the limit, when it exists, is the smallest
positive root of the given polynomial. We have by Theorem \ref{thm:main}
that
\[
a\left(n\right)=n-1-\sum_{i,j\geq1}c_{i,j}a^{i}\left(n-j\right)
\]
for $n\geq1$. Dividing both sides by $n$ and taking limits, we get
\begin{align*}
\alpha:=\lim_{n\rightarrow\infty}\frac{a\left(n\right)}{n} & =\lim_{n\rightarrow\infty}\frac{1}{n}\left(n-1-\sum_{i,j\geq1}\left(r_{i,j}-r_{i-1,j}\right)a^{i}\left(n-j\right)\right)\\
 & =1-\sum_{i,j\geq1}\left(r_{i,j}-r_{i-1,j}\right)\alpha^{i}\\
 & =1-\sum_{i\geq1}r_{i}\alpha^{i}+\sum_{i\geq1}r_{i}\alpha^{i+1}\\
 & =1-r_{1}\alpha-\sum_{i=1}^{k}\left(r_{i+1}-r_{i}\right)\alpha^{i+1}.
\end{align*}
Hence, $\alpha$ is a root of the polynomial
\begin{align*}
q\left(x\right) & :=1-\sum_{i=0}^{k}\left(r_{i+1}-r_{i}\right)x^{i+1}\\
 & =r_{k+1}x^{k+1}+\left(1-x\right)\sum_{i=0}^{k}r_{i}.
\end{align*}
To see that $\alpha$ is the smallest such root, note that $q\left(x\right)$
is the polynomial (\ref{eq:gf_poly}) which appears in the denominator
of $T\left(z\right)$. Considering $T\left(z\right)$ as a complex
rational function, then, observe that $T\left(z\right)$ has a nonzero
radius of convergence $c$. If $u$ is the smallest positive root
of $q\left(x\right)$, then $c\leq u$ since $u$ is a pole of $T\left(z\right)$.
The ratio test implies
\[
\lim_{n\rightarrow\infty}\frac{\left|\symbol T_{n-1}\right|}{\left|\symbol T_{n}\right|}=c.
\]
By Lemma \ref{lem:Prefix}, however, $\left|\symbol T_{n-1}\right|=\left|\symbol T_{n}\right|-\left|\symbol T_{n}\right|_{\symbol 0}=a\left(\left|\symbol T_{n}\right|\right)$.
Hence, the LHS is just $\alpha$.

It is worth noting that $q\left(x\right)$ has a root that is strictly
less than 1 if $r_{1}+r_{2}+\ldots\geq2$. Thus, the only time we
can have $\alpha=1$ is when $r_{1}+r_{2}+\ldots=0$ or $1$. We use
this fact later.

\phantom{}

We next prove that each of the given conditions imply the existence
of the limit. Suppose first that condition (1) holds. If $r_{1}+r_{2}+\ldots=0$,
then $a\left(n\right)=n-1$ and clearly the theorem holds. Otherwise,
suppose $r_{1}+r_{2}+\ldots=1$; that is, there exists some $h\geq1$
such that $r_{h}=1$ and $r_{i}=0$ for all positive $i\neq h$. This
is something of a degenerate case, and the morphism $\sigma$ has
a very simple structure. In particular, if $i$ is an integer written
as $i=qh+r$ for integers $q,r$ with $0\leq r<h$, then 
\[
\sigma^{i}\left(\symbol 0\right)=\left[h\right]^{q}\left[r\right].
\]
It follows that $\left|\sigma^{i}\left(\symbol 0\right)\right|$ is
equal to $\lfloor i/h\rfloor+1$.

Now let $m$ be the smallest integer such that
\[
n\leq\left|\symbol 0\sigma\left(\symbol 0\right)\sigma^{2}\left(\symbol 0\right)\ldots\sigma^{mh-1}\left(\symbol 0\right)\right|.
\]
The above inequality holds if $m=\left\lceil \sqrt{2n/h}\right\rceil $,
since
\[
\left|\symbol 0\sigma\left(\symbol 0\right)\ldots\sigma^{mh-1}\left(\symbol 0\right)\right|=\sum_{i=0}^{mh-1}\left(\lfloor i/h\rfloor+1\right)=\sum_{i=0}^{m-1}\left(i+1\right)h=\frac{hm\left(m+1\right)}{2}\geq\frac{hm^{2}}{2}\geq n.
\]
It follows that
\[
m\leq\left\lceil \sqrt{2n/h}\right\rceil .
\]
 Recall Lemma \ref{lem:Prefix}, which states that
\[
a\left(n\right)=n-\left|\symbol W_{n}\right|_{\symbol 0}
\]
where $\symbol W_{n}$ is the length-$n$ prefix of 
\[
\symbol 0\sigma\left(\symbol 0\right)\sigma^{2}\left(\symbol 0\right)\sigma^{3}\left(\symbol 0\right)\ldots=\symbol r^{-1}\sigma^{\infty}\left(\symbol r\right).
\]
 For any $i\geq0$, $\sigma^{i}\left(\symbol 0\right)$ has at most
one $\symbol 0$, and so 
\begin{align*}
a\left(n\right) & \geq n-\sum_{i=0}^{mh-1}\left|\sigma^{i}\left(\symbol 0\right)\right|_{\symbol 0}\\
 & \geq n-mh\\
 & \geq n-\left\lceil \sqrt{2n/h}\right\rceil h.
\end{align*}
It is clear that $a\left(n\right)/n\leq1$ for all $n$. Thus, 
\[
\liminf_{n\rightarrow\infty}\frac{a\left(n\right)}{n}\geq\liminf_{n\rightarrow\infty}\frac{n-\bigl\lceil\sqrt{2n/h}\bigr\rceil h}{n}=1\geq\limsup_{n\rightarrow\infty}\frac{a\left(n\right)}{n},
\]
and hence the limit exists.

We next show that condition (2) implies the existence of the limit.
For the morphism $\sigma$ defined on $\Sigma_{k}$, the incidence
matrix of $\sigma$ is
\[
M_{\sigma}=\begin{pmatrix}1\\
1 & r_{1} & r_{2} & \cdots & r_{k} & r_{k+1}\\
 & 1\\
 &  & 1\\
 &  &  & \ddots\\
 &  &  &  & 1 & 1
\end{pmatrix}.
\]
Let $\lambda$ denote the Perron-Frobenius eigenvalue of $M_{\sigma}$.
Our goal is to show that $\lambda$ is simple and is the \emph{only}
eigenvalue of $M_{\sigma}$ with modulus $\lambda$. Then we can apply
Proposition \ref{prop:Saari} and Lemma \ref{lem:Prefix} to conclude
\[
\lim_{n\rightarrow\infty}\frac{a\left(n\right)}{n}=\lim_{n\rightarrow\infty}\frac{n-\left|\symbol W_{n}\right|_{\symbol 0}}{n}=1-\lim_{n\rightarrow\infty}\frac{\left|\symbol W_{n}\right|_{\symbol 0}}{n}
\]
is well-defined. To do this, we consider two cases based on the value
of $r_{k+1}$. Define $M_{1}$ and $M_{2}$ to be the two submatrices
of $M_{\sigma}$ as shown:

\[
M_{\sigma}=\left(
\begin{array}{c|ccc}
1   & 0 &   \cdots    & 0 \\ \hline
1   &   &             &   \\
\vdots  &   & M_1 &   \\
0  &   &             &  
\end{array}\right)
=\left(
\begin{array}{c|ccc|c}
1      & 0 &    \cdots    & 0 & 0 \\ \hline
1      &   &              &   & r_{k+1}  \\ 
\vdots &   & M_2          &   & \vdots  \\
0      &   &              &   & 0  \\ \hline
0      & 0 &    \cdots    & 1 & 1
\end{array}\right).
\]If $r_{k+1}\geq1$, then we show that $M_{1}$ is primitive. If $r_{k+1}=0$,
then we show that $M_{2}$ is primitive. In either case, we show that
the eigenvalues of the primitive submatrix in question are also eigenvalues
of $M_{\sigma}$. The remaining eigenvalues of $M_{\sigma}$ are then
shown to be smaller in modulus than $\lambda$, thus proving our goal.
Without loss of generality, we assume $r_{k}\neq r_{k+1}$.

\emph{Case 1: $r_{k+1}\geq1$}. Observe that $M_{1}$ is the indicence
matrix of the morphism $\sigma_{1}$ defined on $\Sigma_{k}\setminus\left\{ \symbol r\right\} $
by:
\begin{align*}
\intsymbol j & \rightarrow\intsymbol{j+1}\symbol 0^{r_{j+1}};\quad0\leq j<k\\
\intsymbol k & \rightarrow\intsymbol k\symbol 0^{r_{k+1}}.
\end{align*}
We show that $M_{1}$ is irreducible. The fact that it is primitive
then follows from the fact that it has nonzero trace; see \cite[Theorem 8.3.9]{allouche_automatic_2003}.
To show that $M_{1}$ is irreducible, it is sufficient to show that
for any $\intsymbol x,\intsymbol y\in\Sigma_{k}\setminus\left\{ \symbol r\right\} $,
there exists an integer $n$ such that $\sigma_{1}^{n}\left(\left[x\right]\right)$
contains at least one $\intsymbol y$. For such $\left[x\right],\left[y\right]$,
the word $\sigma_{1}^{k+1-x}\left(\left[x\right]\right)$ contains
at least one $\symbol 0$ since it is assumed that $r_{k+1}\geq1$.
However, the word $\sigma_{1}^{y}\left(\symbol 0\right)$ also contains
at least one $\left[y\right]$. It follows that the word $\sigma_{1}^{y+k+1-x}\left(\intsymbol x\right)$
must contain at least one $\intsymbol y$, and so $M_{1}$ is irreducible.

The matrix $M_{\sigma}$ is a reducible matrix with two irreducible
matrices along the diagonal: $M_{1}$ and a $1\times1$ matrix each
consisting of a single 1. Thus, the characteristic polynomial of $M_{\sigma}$
is
\[
p_{M_{\sigma}}\left(x\right)=\left(x-1\right)p_{M_{1}}\left(x\right).
\]
Hence, the eigenvalues of $M_{\sigma}$ are the eigenvalues of $M_{1}$
and 1.

\emph{Case} \emph{2: }$r_{k+1}=0$. In this case, $M_{2}$ is the
incidence matrix of the morphism $\sigma_{2}$ defined on $\Sigma_{k-1}\setminus\left\{ \symbol r\right\} $
by:
\begin{align*}
\intsymbol j & \rightarrow\intsymbol{j+1}\symbol 0^{r_{j+1}};\quad0\leq j<k-1\\
\intsymbol{k-1} & \rightarrow\symbol 0^{r_{k}}.
\end{align*}
The argument that $M_{2}$ is irreducible is essentially the same
as in the previous case. For $\intsymbol x,\intsymbol y\in\Sigma_{k-1}\setminus\left\{ \symbol r\right\} $,
$\sigma_{2}^{k-x}\left(\left[x\right]\right)$ contains at least one
$\symbol 0$ since $r_{k}\neq r_{k+1}=0$ and $\sigma_{2}^{y}\left(\symbol 0\right)$
contains at least one $\left[y\right]$. Thus $\sigma_{2}^{y+k-x}\left(\intsymbol x\right)$
must contain at least one $\intsymbol y$. Now one can check that
the characteristic polynomial of $M_{2}$ is equal to
\[
p_{M_{2}}\left(x\right)=x^{k}-r_{1}x^{k-1}-r_{2}x^{k-2}-\ldots-r_{k-1}x-r_{k}.
\]
Thus, the index of imprimitivity of $M_{2}$ is equal to $\gcd\left\{ j\geq1:r_{j}\neq0\right\} $.
By condition (2), however, this is equal to 1. It follows that $M_{2}$
is primitive.

When $r_{k+1}=0$, $M_{\sigma}$ has three irreducible matrices along
the diagonal: $M_{2}$ and two $1\times1$ matrices each consisting
of a single 1. Thus, the characteristic polynomial of $M_{\sigma}$
is
\[
p_{M_{\sigma}}\left(x\right)=\left(x-1\right)^{2}p_{M_{2}}\left(x\right).
\]
As in the previous case, the eigenvalues of $M_{\sigma}$ are the
eigenvalues of $M_{2}$ and 1.

In both cases, to show that $\lambda$ is a simple eigenvalue of $M_{\sigma}$
that dominates all others in modulus it remains to show that $\lambda>1$.
The characteristic polynomial of $M_{\sigma}$ is equal to
\[
p_{M_{\sigma}}\left(x\right)=\left(x-1\right)^{2}\left(x^{k}-r_{1}x^{k-1}-r_{2}x^{k-2}-\ldots-r_{k-1}x-r_{k}\right)-r_{k+1}\left(x-1\right).
\]
Without loss of generality, we may assume that $r_{1}+r_{2}+r_{3}+\ldots\geq2$
since the cases where this sum is 0 or 1 are treated separately in
condition (1). If $r_{k+1}\geq1$, then $p_{M_{\sigma}}'\left(1\right)=-r_{k+1}<0$.
If $r_{k+1}=0$, then $p_{M_{\sigma}}'\left(1\right)=0$ but 
\[
p_{M_{\sigma}}''\left(1\right)=2\left(1-r_{1}-r_{2}-\ldots-r_{k}\right)=2\left(1-r_{1}-r_{2}-\ldots\right)<0.
\]
In either case, there exists $\epsilon>0$ so that $p_{M_{\sigma}}\left(1+\epsilon\right)<0$.
Since $p_{M_{\sigma}}\left(x\right)$ is monic it must be positive
for $x$ large enough. By the intermediate value theorem, therefore,
there exists $x>1$ which is an eigenvalue of $M_{\sigma}$. It follows
that $\lambda\geq x$. 

\phantom{}

Finally, we prove the converse of the theorem, that is, we show that
the the existence of the limit implies either condition (1) or (2).
Suppose $\alpha:=\lim a\left(n\right)/n$ exists, and consider the
integer sequence
\[
R'=\left\langle s,r_{h},r_{2h},\ldots\right\rangle 
\]
where $h=\gcd\left\{ j\geq1:r_{j}\geq1\right\} $. Note that if $r_{1}+r_{2}+\ldots=0$,
then $h$ is not defined but the theorem holds trivially. Let $\gamma:=\sigma_{R'}$
(as before, we write $\sigma=\sigma_{R}$). We use the following observations:
\begin{enumerate}
\item Any symbol in $\sigma^{n}\left(\symbol 0\right)$ other than $\intsymbol k$
is congruent to $n\mod h$. In particular, $\left|\sigma^{n}\left(\symbol 0\right)\right|_{\symbol 0}=0$
if $n\not\not\equiv0\mod h$.
\item If $x\not\equiv-1\mod h$, then $r_{x+1}=0$ and so $\left|\sigma\left(\intsymbol x\right)\right|=\left|\intsymbol{x+1}\right|=1$.
It follows that for all $n\geq0$ and $0\leq j<h$,
\[
\left|\sigma^{nh+j}\left(\symbol 0\right)\right|=\left|\sigma^{nh}\left(\symbol 0\right)\right|.
\]

\item We also have that $\sigma^{nh}\left(\symbol 0\right)=\gamma^{n}\left(\symbol 0\right)$
after relabelling each symbol $\intsymbol x$ in $\sigma^{hn}\left(\symbol 0\right)$
to $\intsymbol{x/h}$. In particular, for all $n\geq0$,
\[
\left|\sigma^{nh}\left(\symbol 0\right)\right|=\left|\gamma^{n}\left(\symbol 0\right)\right|\quad\mbox{and}\quad\left|\sigma^{nh}\left(\symbol 0\right)\right|_{\symbol 0}=\left|\gamma^{n}\left(\symbol 0\right)\right|_{\symbol 0}.
\]

\item For $n\geq1$,
\begin{align*}
\sigma^{n}\left(\symbol r\right) & =\symbol{r0}\sigma\left(\symbol 0\right)\sigma^{2}\left(\symbol 0\right)\ldots\sigma^{n-1}\left(\symbol 0\right)\\
\gamma^{n}\left(\symbol r\right) & =\symbol{r0}\gamma\left(\symbol 0\right)\gamma^{2}\left(\symbol 0\right)\ldots\gamma^{n-1}\left(\symbol 0\right).
\end{align*}

\item For $n\geq0$,
\begin{align*}
\left|\gamma^{n}\left(\symbol r\right)\right| & =\left|\gamma^{n+1}\left(\symbol r\right)\right|-\left|\gamma^{n+1}\left(\symbol r\right)\right|_{\symbol 0}.
\end{align*}

\end{enumerate}
\begin{figure}
\input{asymptotics.tex}\caption{The tree $\tree T_{R}$ representing $R=\left\langle 1,0,2,0,0,\ldots\right\rangle $.
In this example, $h=2$. The tree $\tree T_{R'}$ is depicted in Figure
\ref{fig:Conolly}. }
\end{figure}

\end{proof}
\begin{figure}
\includegraphics[width=1\columnwidth]{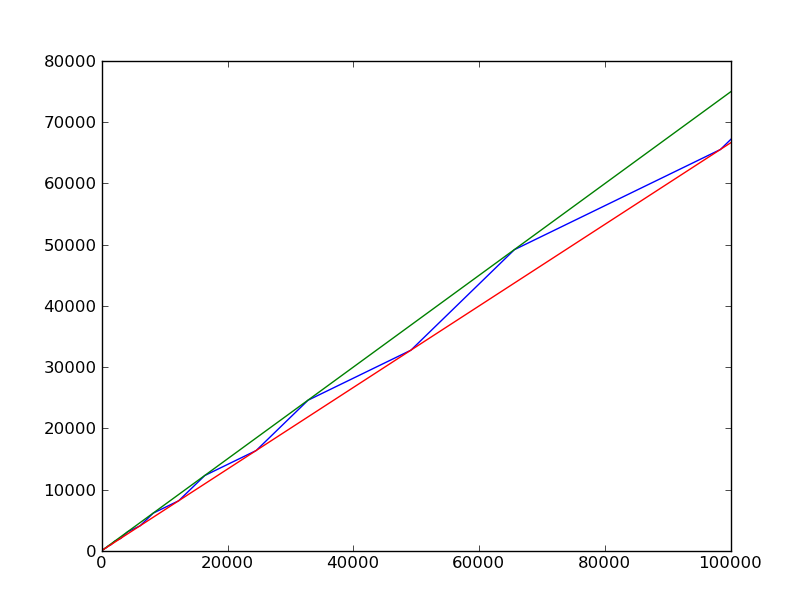}\caption{A plot of $a_{R}\left(n\right)$ when $R=\left\langle 1,0,2,0,0,\ldots\right\rangle $.
The sequence is bounded by the two lines shown, but it has no limit.}
\end{figure}

\begin{proof}
Because $\alpha$ exists, we must have that
\[
1-\alpha=\lim_{n\rightarrow\infty}\frac{\left|\sigma^{nh}\left(\symbol r\right)\right|_{\symbol 0}}{\left|\sigma^{nh}\left(\symbol r\right)\right|}=\lim_{n\rightarrow\infty}\frac{\left|\sigma^{nh+1}\left(\symbol r\right)\right|_{\symbol 0}}{\left|\sigma^{nh+1}\left(\symbol r\right)\right|}.
\]
Let $\beta=\lim_{n\rightarrow\infty}a_{R'}\left(n\right)/n$, which
exists since $\gcd\left\{ hj\geq1:r_{hj}\geq1\right\} =1$. We compute
both of these limits in terms of $\beta$ and $h$, and conclude that
at least one of $\beta,h$ must equal 1. If $\beta=1$, then $\alpha=1$
and so condition (1) holds by previous remarks. Assume therefore that
$\beta<1$.

By observations (4), (2), (3), and (4),
\begin{align*}
\left|\sigma^{nh}\left(\symbol r\right)\right|=1+\sum_{i=0}^{nh-1}\left|\sigma^{i}\left(\symbol 0\right)\right|=1+h\sum_{i=0}^{n-1}\left|\sigma^{ih}\left(\symbol 0\right)\right| & =1+h\sum_{i=0}^{n-1}\left|\gamma^{i}\left(\symbol 0\right)\right|\\
 & =1+h\left(\left|\gamma^{n}\left(\symbol r\right)\right|-1\right)
\end{align*}
and by observations (4), (1), (3), and (4),
\begin{align*}
\left|\sigma^{nh}\left(\symbol r\right)\right|_{\symbol 0}=\sum_{i=0}^{nh-1}\left|\sigma^{i}\left(\symbol 0\right)\right|_{\symbol 0}=\sum_{i=0}^{n-1}\left|\sigma^{ih}\left(\symbol 0\right)\right|_{\symbol 0} & =\sum_{i=0}^{n-1}\left|\gamma^{i}\left(\symbol 0\right)\right|_{\symbol 0}=\left|\gamma^{n}\left(\symbol r\right)\right|_{\symbol 0}.
\end{align*}
Thus
\[
\lim_{n\rightarrow\infty}\frac{\left|\sigma^{nh}\left(\symbol r\right)\right|_{\symbol 0}}{\left|\sigma^{nh}\left(\symbol r\right)\right|}=\lim_{n\rightarrow\infty}\frac{\left|\gamma^{n}\left(\symbol r\right)\right|_{\symbol 0}}{1+h\left(\left|\gamma^{n}\left(\symbol r\right)\right|-1\right)}=\frac{1-\beta}{h}.
\]
On the other hand, 
\begin{align*}
\left|\sigma^{nh+1}\left(\symbol r\right)\right| & =\left|\sigma^{nh}\left(\symbol r\right)\right|+\left|\sigma^{nh}\left(\symbol 0\right)\right| & (4)\\
 & =1+h\left(\left|\gamma^{n}\left(\symbol r\right)\right|-1\right)+\left|\gamma^{n}\left(\symbol 0\right)\right| & (3)\\
 & =1+\left(h-1\right)\left|\gamma^{n}\left(\symbol r\right)\right|-h+\left|\gamma^{n}\left(\symbol r\right)\right|+\left|\gamma^{n}\left(\symbol 0\right)\right|\\
 & =1+\left(h-1\right)\left|\gamma^{n}\left(\symbol r\right)\right|-h+\left|\gamma^{n+1}\left(\symbol r\right)\right| & (4)\\
 & =\left(h-1\right)\left(\left|\gamma^{n}\left(\symbol r\right)\right|-1\right)+\left|\gamma^{n+1}\left(\symbol r\right)\right|\\
 & =\left(h-1\right)\left(\left|\gamma^{n+1}\left(\symbol r\right)\right|-\left|\gamma^{n+1}\left(\symbol r\right)\right|_{\symbol 0}-1\right)+\left|\gamma^{n+1}\left(\symbol r\right)\right| & (5)\\
 & =h\left|\gamma^{n+1}\left(\symbol r\right)\right|-\left(h-1\right)\left(\left|\gamma^{n+1}\left(\symbol r\right)\right|_{\symbol 0}-1\right),
\end{align*}
and by observations (4), (3), and (4),
\[
\left|\sigma^{nh+1}\left(\symbol r\right)\right|_{\symbol 0}=\left|\sigma^{nh}\left(\symbol r\right)\right|_{\symbol 0}+\left|\sigma^{nh}\left(\symbol 0\right)\right|_{\symbol 0}=\left|\gamma^{n}\left(\symbol r\right)\right|_{\symbol 0}+\left|\gamma^{n}\left(\symbol 0\right)\right|_{\symbol 0}=\left|\gamma^{n+1}\left(\symbol r\right)\right|_{\symbol 0}.
\]
Therefore,
\begin{align*}
\lim_{n\rightarrow\infty}\frac{\left|\sigma^{nh+1}\left(\symbol r\right)\right|_{\symbol 0}}{\left|\sigma^{nh+1}\left(\symbol r\right)\right|} & =\lim_{n\rightarrow\infty}\left(\frac{h\left|\gamma^{n+1}\left(\symbol r\right)\right|-\left(h-1\right)\left(\left|\gamma^{n+1}\left(\symbol r\right)\right|_{\symbol 0}-1\right)}{\left|\gamma^{n+1}\left(\symbol r\right)\right|_{\symbol 0}}\right)^{-1}\\
 & =\lim_{n\rightarrow\infty}\left(\frac{h\left|\gamma^{n+1}\left(\symbol r\right)\right|}{\left|\gamma^{n+1}\left(\symbol r\right)\right|_{\symbol 0}}-\frac{\left(h-1\right)\left(\left|\gamma^{n+1}\left(\symbol r\right)\right|_{\symbol 0}-1\right)}{\left|\gamma^{n+1}\left(\symbol r\right)\right|_{\symbol 0}}\right)^{-1}\\
 & =\left(\frac{h}{1-\beta}-\left(h-1\right)\right)^{-1}\\
 & =\frac{1-\beta}{h-\left(h-1\right)\left(1-\beta\right)}.
\end{align*}
The two limits are equal, however, and since $\beta\neq1$ we must
have that the two denominators are equal, that is,
\[
h=h-\left(h-1\right)\left(1-\beta\right).
\]
But this implies $\left(h-1\right)=\beta\left(h-1\right)$, and so
$h=1$ which is precisely condition (2). 
\end{proof}

\section{Concluding Remarks}

An observation one can make in the study of nested recurrence relations
is that they tend to follow come in one of two flavours: either they
are highly chaotic and unpredictable, or they appear to have some
regular structure. Recurrences of the latter form quite often admit
some sort of combinatorial interpretation, and the ones studied in
this paper are no exception. Using our morphism and tree interpretations,
we are able to better understand the behaviour of our mysterious recurrences.
The picture remains incomplete, however, since our asymptotic analysis
only applies when $a_{R}\left(n\right)$ has only finitely many terms.
The general case remains unsolved, but we suspect that our asymptotic
results no longer hold in the general case and the situation becomes
more subtle.

An interesting open problem is to determine whether or not these recurrences
have closed forms, since to date very few nested recurrences have
been found with closed-form solutions (for some recent developments
on this front see \cite{isgur_combinatorial_2012}). Moreover, not
all morphic nested recurrences are of the form (\ref{eq:main_thm}).
Indeed, Hofstadter noted in \cite{hofstadter_douglas_r._godel_1979}
that his ``married'' functions (see \cite{stoll_hofstadters_2008})
have a morphism interpretation as well. It would be interesting to
see if Theorem \ref{thm:main} could be generalized to include mutually
defined nested recurrences. What would such morphisms look like?

\bibliographystyle{plain}
\bibliography{morphisms}

\end{document}

%% file: fibtree_symbols.tex
\tikzset{edge from parent/.style={draw, edge from parent path={(\tikzparentnode) -- (\tikzchildnode)}}}   
\tikzset{level distance=23pt}
\begin{tikzpicture}[nodes={inner sep=0pt,minimum height=12pt,minimum width=12pt,font=\scriptsize},sibling distance=3pt]

\Tree [.\node[root]{$\symbol{r}$};
          [.\node[root]{$\symbol{r}$};
              [.\node[root]{$\symbol{r}$};
                  [.\node[root]{$\symbol{r}$};
                      [.\node[root]{$\symbol{r}$};
                          [.\node[root]{$\symbol{r}$};
                              [.\node[root]{$\symbol{r}$}; ]
                              [.\node[nonsquare]{$\symbol{0}$}; ] ]
                          [.\node[nonsquare]{$\symbol{0}$};
                              [.\node[square]{$\symbol{1}$}; ] ] ]
                      [.\node[nonsquare]{$\symbol{0}$};
                          [.\node[square]{$\symbol{1}$};
                              [.\node[square]{$\symbol{2}$}; ]
                              [.\node[nonsquare]{$\symbol{0}$}; ] ] ] ]
                  [.\node[nonsquare]{$\symbol{0}$};
                      [.\node[square]{$\symbol{1}$};
                          [.\node[square]{$\symbol{2}$};
                              [.\node[square]{$\symbol{3}$}; ]
                              [.\node[nonsquare]{$\symbol{0}$}; ] ]
                          [.\node[nonsquare]{$\symbol{0}$};
                              [.\node[square]{$\symbol{1}$}; ] ] ] ] ]
              [.\node[nonsquare]{$\symbol{0}$};
                  [.\node[square]{$\symbol{1}$};
                      [.\node[square]{$\symbol{2}$};
                          [.\node[square]{$\symbol{3}$};
                              [.\node[square]{$\symbol{4}$}; ]
                              [.\node[nonsquare]{$\symbol{0}$}; ] ]
                          [.\node[nonsquare]{$\symbol{0}$};
                              [.\node[square]{$\symbol{1}$}; ] ] ]
                      [.\node[nonsquare]{$\symbol{0}$};
                          [.\node[square]{$\symbol{1}$};
                              [.\node[square]{$\symbol{2}$}; ]
                              [.\node[nonsquare]{$\symbol{0}$}; ] ] ] ] ] ]
          [.\node[nonsquare]{$\symbol{0}$};
              [.\node[square]{$\symbol{1}$};
                  [.\node[square]{$\symbol{2}$};
                      [.\node[square]{$\symbol{3}$};
                          [.\node[square]{$\symbol{4}$};
                              \node[square]{$\symbol{5}$}; 
                              \node[nonsquare]{$\symbol{0}$}; ]
                          [.\node[nonsquare]{$\symbol{0}$};
                              \node[square]{$\symbol{1}$}; ] ]
                      [.\node[nonsquare]{$\symbol{0}$};
                          [.\node[square]{$\symbol{1}$};
                              \node[square]{$\symbol{2}$}; 
                              \node[nonsquare]{$\symbol{0}$}; ] ] ]
                  [.\node[nonsquare]{$\symbol{0}$};
                      [.\node[square]{$\symbol{1}$};
                          [.\node[square]{$\symbol{2}$};
                              \node[square]{$\symbol{3}$};
                              \node[nonsquare]{$\symbol{0}$};  ]
                          [.\node[nonsquare]{$\symbol{0}$};
                              \node[square]{$\symbol{1}$}; ] ] ] ] ] ]
\end{tikzpicture}

%% file: root2tree.tex
\tikzset{edge from parent/.style={draw, edge from parent path={(\tikzparentnode) -- (\tikzchildnode)}}}   
\tikzset{level distance=23pt}
\begin{tikzpicture}[nodes={inner sep=0pt,minimum height=11pt,minimum width=11pt,font=\tiny},sibling distance=1pt]

\Tree   [.\node[root]{0}; 
            [.\node[root]{0};
                [.\node[root]{0};
                    [.\node[root]{0};
                        [.\node[root]{0}; ]
                        [.\node[nonsquare]{1}; ] ]
                    [.\node[nonsquare]{1};
                        [.\node[square]{2}; ]
                        [.\node[nonsquare]{3}; ] ] ]
                [.\node[nonsquare]{1};
                    [.\node[square]{2};
                        [.\node[square]{4}; ]
                        [.\node[nonsquare]{5}; ] 
                        [.\node[nonsquare]{6}; ] ]
                    [.\node[nonsquare]{3};
                        [.\node[square]{7}; ]
                        [.\node[nonsquare]{8}; ] ] ] ]
            [.\node[nonsquare]{1}; 
                [.\node[square]{2};
                    [.\node[square]{4};
                        [.\node[square]{9}; ]
                        [.\node[nonsquare]{10}; ] 
                        [.\node[nonsquare]{11}; ] ]
                    [.\node[nonsquare]{5};
                        [.\node[square]{12}; ]
                        [.\node[nonsquare]{13}; ] ] 
                    [.\node[nonsquare]{6};
                        [.\node[square]{14}; ]
                        [.\node[nonsquare]{15}; ] ] ]
                [.\node[nonsquare]{3};
                    [.\node[square]{7};
                        [.\node[square]{16}; ]
                        [.\node[nonsquare]{17}; ] 
                        [.\node[nonsquare]{18}; ] ]
                    [.\node[nonsquare]{8};
                        [.\node[square]{19}; ]
                        [.\node[nonsquare]{20}; ] ] ] ] ]
\end{tikzpicture}

%% file: conollytree.tex
\tikzset{edge from parent/.style={draw, edge from parent path={(\tikzparentnode) -- (\tikzchildnode)}}}   
\tikzset{level distance=22pt}
\begin{tikzpicture}[nodes={inner sep=0pt,minimum height=10pt,minimum width=10pt,font=\tiny},sibling distance=1pt]

\Tree   [.\node[root]{0};
            [.\node[root]{$\tree{T}$}; ]
            [.\node[nonsquare]{1}; 
                [.\node[square]{2}; 
                    [.\node[square]{5}; 
                        [.\node[square]{12}; \node[square]{27}; ] ] ]
                [.\node[nonsquare]{3}; 
                    [.\node[square]{6}; 
                        [.\node[square]{13}; \node[square]{28}; ] ]
                    [.\node[nonsquare]{7}; 
                        [.\node[square]{14}; \node[square]{29}; ]
                        [.\node[nonsquare]{15}; \node[square]{30}; \node[nonsquare]{31}; \node[nonsquare]{32}; ]
                        [.\node[nonsquare]{16}; \node[square]{33}; \node[nonsquare]{34}; \node[nonsquare]{35}; ] ]
                    [.\node[nonsquare]{8}; 
                        [.\node[square]{17}; \node[square]{36}; ]
                        [.\node[nonsquare]{18}; \node[square]{37}; \node[nonsquare]{38}; \node[nonsquare]{39}; ]
                        [.\node[nonsquare]{19}; \node[square]{40}; \node[nonsquare]{41}; \node[nonsquare]{42}; ] ] ]
                [.\node[nonsquare]{4};
                    [.\node[square]{9}; 
                        [.\node[square]{20}; \node[square]{43}; ] ]
                    [.\node[nonsquare]{10}; 
                        [.\node[square]{21}; \node[square]{44}; ]
                        [.\node[nonsquare]{22}; \node[square]{45}; \node[nonsquare]{46}; \node[nonsquare]{47}; ]
                        [.\node[nonsquare]{23}; \node[square]{48}; \node[nonsquare]{49}; \node[nonsquare]{50}; ] ]
                    [.\node[nonsquare]{11}; 
                        [.\node[square]{24}; \node[square]{51}; ]
                        [.\node[nonsquare]{25}; \node[square]{52}; \node[nonsquare]{53}; \node[nonsquare]{54}; ]
                        [.\node[nonsquare]{26}; \node[square]{55}; \node[nonsquare]{56}; \node[nonsquare]{57}; ] ] ] ] ]
\end{tikzpicture}

%% file: second_lemma.tex
\tikzset{edge from parent/.style={draw, edge from parent path={(\tikzparentnode) -- (\tikzchildnode)}}}   
\tikzset{level distance=23pt}
\begin{tikzpicture}[nodes={inner sep=3pt,minimum height=12pt,minimum width=12pt,font=\scriptsize},
sibling distance=.2cm]

\Tree   [.\node[maybesquare](rt){$\cdot$};
            [.\node[maybesquare](v0){$\cdot$};
                [.\node[maybesquare](v1){$\cdot$}; 
                    [.\node[maybesquare](v2){$\cdot$};
                        [.\node[maybesquare](v3){$\cdot$}; ] ] ] ]
            [.\node[nonsquare](w0){$\symbol{0}$};
                [.\node[square](w1){$\symbol{1}$}; 
                    [.\node[square](w2){$\symbol{2}$}; 
                        [.\node[square](w3){$\symbol{3}$}; ]
                        [.\node[nonsquare]{$\symbol{0}$}; ]
                        [.\node[nonsquare]{$\symbol{0}$}; ] ] ] ] ]
                    
\node[right of=rt,minimum width=2.2cm, xshift=.45cm, text width=2.2cm,align=left]{$\parent w{4}=\parent v{4}$};

\node[right of=w0,minimum width=2.2cm, xshift=.45cm, text width=2.2cm,align=left]{$\parent v{3}$};
\node[right of=w1,minimum width=2.2cm, xshift=.45cm, text width=2.2cm,align=left]{$\parent v{2}$};
\node[right of=w2,minimum width=2.2cm, xshift=.45cm, text width=2.2cm,align=left]{$\parent v{1}$};
\node[below of=w3,yshift=.6cm]{$v$};

\node[left of=v0,minimum width=2.2cm, xshift=-.45cm, text width=2.2cm,align=right]{$\parent w{3}$};
\node[left of=v1,minimum width=2.2cm, xshift=-.45cm, text width=2.2cm,align=right]{$\parent w{2}$};
\node[left of=v2,minimum width=2.2cm, xshift=-.45cm, text width=2.2cm,align=right]{$\parent w{1}$};
\node[left of=v3,minimum width=2.2cm, xshift=-.45cm, text width=2.2cm,align=right]{$w$};

\end{tikzpicture}

%% file: asymptotics.tex
\tikzset{edge from parent/.style={draw, edge from parent path={(\tikzparentnode) -- (\tikzchildnode)}}}   
\tikzset{level distance=16pt}
\begin{tikzpicture}[nodes={inner sep=0pt,minimum height=10pt,minimum width=10pt,font=\tiny},sibling distance=3pt]

\Tree   [.\node[root]{0};
            [.\node[root]{$\tree{T}_R$}; ]
			[.\node[nonsquare]{1};
				[.\node[square]{2}; 
					[.\node[square]{3}; 
						[.\node[square]{6}; 
							[.\node[square]{9}; 
								[.\node[square]{16}; 
									[.\node[square]{23}; 
										[.\node[square]{38}; ] ] ] ] ] ]
					[.\node[nonsquare]{4}; 
						[.\node[square]{7}; 
							[.\node[square]{10}; 
								[.\node[square]{17}; 
									[.\node[square]{24}; 
										[.\node[square]{39}; ] ] ] ]
							[.\node[nonsquare]{11}; 
								[.\node[square]{18}; 
									[.\node[square]{25}; 
										[.\node[square]{40}; ] ]
									[.\node[nonsquare]{26}; 
										[.\node[square]{41}; ] ]
									[.\node[nonsquare]{27}; 
										[.\node[square]{42}; ] ] ] ]
							[.\node[nonsquare]{12}; 
								[.\node[square]{19}; 
									[.\node[square]{28}; 
										[.\node[square]{43}; ] ]
									[.\node[nonsquare]{29}; 
										[.\node[square]{44}; ] ]
									[.\node[nonsquare]{30}; 
										[.\node[square]{45}; ] ] ] ] ] ]
					[.\node[nonsquare]{5}; 
						[.\node[square]{8};
							[.\node[square]{13};
								[.\node[square]{20}; 
									[.\node[square]{31}; 
										[.\node[square]{46}; ] ] ] ]
							[.\node[nonsquare]{14};
								[.\node[square]{21}; 
									[.\node[square]{32}; 
										[.\node[square]{47}; ] ]
									[.\node[nonsquare]{33}; 
										[.\node[square]{48}; ] ]
									[.\node[nonsquare]{34}; 
										[.\node[square]{49}; ] ] ] ]
							[.\node[nonsquare]{15};
								[.\node[square]{22}; 
									[.\node[square]{35}; 
										[.\node[square]{50}; ] ]
									[.\node[nonsquare]{36}; 
										[.\node[square]{51}; ] ]
									[.\node[nonsquare]{37}; 
										[.\node[square]{52}; ] ] ] ] ] ] ] ] ]
\end{tikzpicture}

%% file: morphisms.bbl
\begin{thebibliography}{10}

\bibitem{allouche_automatic_2003}
J.~P. Allouche and J.~Shallit.
\newblock {\em Automatic sequences: theory, applications, generalizations}.
\newblock Cambridge university press, 2003.

\bibitem{burton_curious_1986}
Robert~P. Burton and Douglas~M. Campbell.
\newblock A curious recursive function.
\newblock {\em International Journal of Computer Mathematics},
  19(3-4):245--257, 1986.

\bibitem{downey_family_1982}
P.~J Downey and R.~E Griswold.
\newblock On a family of nested recurrences.
\newblock {\em Fibonacci Quarterly}, 22(4):310--317, 1982.

\bibitem{golomb_problem_1966}
S.~W. Golomb.
\newblock Problem 5407.
\newblock {\em Amer. Math. Monthly}, 73(674):74, 1966.

\bibitem{graham_concrete_1994}
R.~L Graham, D.~E Knuth, and O.~Patashnik.
\newblock {\em Concrete mathematics: a foundation for computer science},
  volume~2.
\newblock {Addison-Wesley} Reading, {MA}, 1994.

\bibitem{granville_strange_1988}
V.~Granville and J.~P Rasson.
\newblock A strange recursive relation.
\newblock {\em Journal of Number Theory}, 30(2):238--241, 1988.

\bibitem{hofstadter_douglas_r._godel_1979}
Douglas~R. Hofstadter.
\newblock {\em G\"odel, Escher, Bach: an Eternal Golden Braid}.
\newblock Basic Books, 1979.

\bibitem{isgur_combinatorial_2012}
A.~Isgur, V.~Kuznetsov, and S.~Tanny.
\newblock A combinatorial approach for solving certain nested recursions with
  non-slow solutions.
\newblock {\em Arxiv preprint {arXiv:1202.0276}}, 2012.

\bibitem{isgur_solving_2011}
A.~Isgur, M.~Rahman, and S.~Tanny.
\newblock Solving non-homogeneous nested recursions using trees.
\newblock {\em Arxiv preprint {arXiv:1105.2351}}, 2011.

\bibitem{kiss_generalization_1992}
P.~Kiss and B.~Zay.
\newblock On a generalization of a recursive sequence.
\newblock {\em Fibonacci Quart}, 30:103--109, 1992.

\bibitem{kubo_conways_1996}
T.~Kubo and R.~Vakil.
\newblock On {C}onway's recursive sequence.
\newblock {\em Discrete Mathematics}, 152(1):225--252, 1996.

\bibitem{meyer_matrix_2000}
C.~Meyer.
\newblock {\em Matrix analysis and applied linear algebra}.
\newblock Number~71. Society for Industrial Mathematics, 2000.

\bibitem{rahman_combinatorial_2011}
M.~Rahman.
\newblock A combinatorial interpretation of {H}ofstadter's {$G$}-sequence.
\newblock {\em Arxiv preprint {arXiv:1105.1718}}, 2011.

\bibitem{ruskey_combinatorics_2009}
F.~Ruskey and C.~Deugau.
\newblock The combinatorics of certain $k$-ary {meta-Fibonacci} sequences.
\newblock {\em Journal of Integer Sequences}, 12(2):3, 2009.

\bibitem{saari_frequency_2006}
K.~Saari.
\newblock On the frequency of letters in morphic sequences.
\newblock {\em Computer Science-Theory and Applications}, pages 334--345, 2006.

\bibitem{oeis}
N.~J.~A. Sloane.
\newblock Online encyclopedia of integer sequences.
\newblock \url{http://oeis.org}.

\bibitem{stolarsky_beatty_1976}
K.~B Stolarsky.
\newblock Beatty sequences, continued fractions, and certain shift operators.
\newblock {\em Canad. Math. Bull}, 19(4):473--482, 1976.

\bibitem{stoll_hofstadters_2008}
T.~Stoll.
\newblock On {H}ofstadter's married functions.
\newblock {\em Fibonacci Quarterly}, 46(47):2009, 2008.

\bibitem{vajda_fibonacci_1989}
S.~Vajda.
\newblock {\em Fibonacci \& Lucas numbers, and the golden section: Theory and
  applications}.
\newblock Ellis Horwood, 1989.

\end{thebibliography}
